\documentclass[12pt]{amsart}
\usepackage[margin=1in]{geometry}

\usepackage{aliascnt}
\usepackage{amssymb}
\usepackage{amsmath}
\usepackage{enumitem}
\usepackage{xcolor}
\definecolor{darkblue}{rgb}{0.0, 0.0, 0.8}
\usepackage[colorlinks=true, allcolors=darkblue]{hyperref}
\usepackage{tikz}
\usepackage{tikz-3dplot}
\usepackage{tikz-cd}
\usepackage{pgfplots}
\usepackage{mathtools}
\usepackage{todonotes}
\usepackage{tabularx}
\usepackage{mathtools}
\usepackage{aliascnt}
\usepackage[colorlinks=true, allcolors=darkblue]{hyperref}
\usepackage{comment}
\newtheorem{theorem}{Theorem}[section]

\newaliascnt{lemma}{theorem}
\newtheorem{lemma}[lemma]{Lemma}
\aliascntresetthe{lemma}

\newaliascnt{proposition}{theorem}
\newtheorem{proposition}[proposition]{Proposition}
\aliascntresetthe{proposition}

\newaliascnt{corollary}{theorem}
\newtheorem{corollary}[corollary]{Corollary}
\aliascntresetthe{corollary}

\newaliascnt{conjecture}{theorem}

\aliascntresetthe{conjecture}

\newaliascnt{assumption}{theorem}

\aliascntresetthe{assumption}

\newaliascnt{definition}{theorem}
\newtheorem{definition}[definition]{Definition}
\aliascntresetthe{definition}

\newaliascnt{question}{theorem}

\aliascntresetthe{question}

\newaliascnt{remark}{theorem}
\newtheorem{remark}[remark]{Remark}
\aliascntresetthe{remark}

\newaliascnt{example}{theorem}

\aliascntresetthe{example}

\newtheorem*{notation*}{Notation}
\newtheorem*{theorem*}{Theorem}
\newtheorem*{conjecture*}{Conjecture}

\numberwithin{figure}{section}
\numberwithin{table}{section}
\numberwithin{equation}{section}

\allowdisplaybreaks
\def\S{\mathbb{S}}

\def\NN{\mathbb{N}}

\def\RR{\mathbb{R}}
\def\CC{\mathbb{C}}
\def\Slr{\mathbb{S}\left(\ell_{\RR}^2 \right)}

\def\U{\mathcal{U}}
\def\Glt{\mathcal{G}^{\ell^2}}
\def\Gfr{\mathcal{G}^{\mathrm{FR}}}
\def\i{\mathrm{i}}

\def\d{\mathrm{d}}

\newcommand{\R}{\mathbb{R}}

\newcommand\restr[2]{\ensuremath{\left.#1\right|_{#2}}}

\usepackage[capitalise,nameinlink]{cleveref}
\crefname{theorem}{theorem}{theorems}
\Crefname{theorem}{Theorem}{Theorems}
\crefname{lemma}{lemma}{lemmas}
\Crefname{lemma}{Lemma}{Lemmas}
\crefname{proposition}{proposition}{propositions}
\Crefname{proposition}{Proposition}{Propositions}
\crefname{corollary}{corollary}{corollaries}
\Crefname{corollary}{Corollary}{Corollaries}
\crefname{conjecture}{conjecture}{conjectures}
\Crefname{conjecture}{Conjecture}{Conjectures}
\crefname{assumption}{assumption}{assumptions}
\Crefname{assumption}{Assumption}{Assumptions}
\crefname{definition}{definition}{definitions}
\Crefname{definition}{Definition}{Definitions}
\crefname{question}{question}{questions}
\Crefname{question}{Question}{Questions}
\crefname{remark}{remark}{remarks}
\Crefname{remark}{Remark}{Remarks}
\crefname{example}{example}{examples}
\Crefname{example}{Example}{Examples}
\begin{document}

\title[Information Geometry via the $Q$-Root Transform]{Information Geometry via the $Q$-Root Transform}

\author{Levin Maier}
\address{Faculty of Mathematics and Computer Science,
	University of Heidelberg}
\email{lmaier@mathi.uni-heidelberg.de}

\begin{abstract}
In this paper, we introduce \emph{$\ell^p$-information geometry}, an infinite-dimensional framework that shares key features with the geometry of the space of probability densities \( \mathrm{Dens}(M) \) on a closed manifold, while also incorporating aspects of measure-valued information geometry. We define the \emph{$\ell^2$-probability simplex} with a noncanonical differentiable structure induced via the \emph{$q$-root transform} from an open subset of the \( \ell^q \)-sphere. This choice makes the \(q\)-root transform an \emph{isometry} and allows us to construct the \(\ell^2\)- and \(\ell^q\)-Fisher--Rao geometries, including \emph{Amari--\v{C}encov \(\alpha\)-connections} and a \emph{Chern connection} in the \(\ell^q\)-setting.

We then apply this framework to an infinite-dimensional linear optimization problem. We show that the corresponding gradient flow with respect to the \(\ell^2\)--Fisher--Rao metric can be solved explicitly, converges to a maximizer under a natural monotonicity assumption, and admits an interpretation as the geodesic flow of an \emph{exponential connection}. In particular, we prove that this \(e\)-connection is \emph{geodesically complete}. We further relate these flows to a \emph{completely integrable Hamiltonian system} through a \emph{momentum map} associated with a Hamiltonian torus action on infinite-dimensional complex projective space.

Finally, inspired by the \(\ell^2\)-theory, we outline an analogous Fisher--Rao geometry for \( \mathrm{Dens}(M) \) on possibly noncompact Riemannian manifolds, showing that, with a suitable spherical differentiable structure, the square-root transform remains an \emph{isometry}.
\end{abstract}

\maketitle
\setcounter{tocdepth}{1}
\tableofcontents

\section{Introduction}\label{sec:introduction}
Infinite-dimensional information geometry has by now developed into an active area of research. In particular, substantial work has been carried out on the geometry of the space of smooth probability densities \(\mathrm{Dens}(M)\) on a closed Riemannian manifold equipped with the Fisher--Rao metric; see \cite{FisherRaoisunique,bauerLpFisherRaoMetricAmari2024,khesin_GAFA,khesin_geometry_2019, khesin2024informationgeometrydiffeomorphismgroups} and the references therein. Another natural infinite-dimensional generalization of finite-dimensional information geometry, as studied in \cite{MR1800071,InfGeo_book} and the references therein, is the Fisher--Rao geometry of probability simplices with infinitely many coordinate entries.

A canonical model for such a geometry is the probability simplex in \(\ell^2_{\mathbb{R}}\), called the \(\ell^2\)-\emph{probability simplex} introduced in \Cref{s: 2}. A first basic observation of the present article is that, when this simplex is equipped with its canonical differentiable structure inherited from the ambient space, the Fisher--Rao metric fails to be well defined. To overcome this difficulty, we pull back, via the \(q\)-root transform \eqref{e: def q-root transform}, the differentiable structure of an open subset of the \(\ell^q\)-sphere to the probability simplex in \(\ell^2\). This yields a differentiable structure that is strong enough to make the Fisher--Rao geometry well behaved, and in particular allows us to define the \(\ell^2\)-Fisher--Rao metric \eqref{e:defi_l2_Fisher_Rao} as well as the \(\ell^q\)-Fisher--Rao metric \eqref{e: def q-Fisher-Rao metric} on the \(\ell^2\)-probability simplex. The latter may be viewed as the \(\ell^q\)-analogue of the \(L^q\)-Fisher--Rao metric introduced in \cite{bauerLpFisherRaoMetricAmari2024}.

Our \textbf{first main contribution} is thus the development of \(\ell^2\)- and \(\ell^q\)-information geometry on the infinite-dimensional probability simplex. More precisely, we show that the square-root transform and, more generally, the \(q\)-root transform identify these geometries isometrically with open subsets of the corresponding \(\ell^q\)-spheres; see \Cref{t: square root map is isometry} for the \(\ell^2\)-case and \Cref{T: q-root transform as isom} for the \(\ell^q\)-case. These results may be regarded as \(\ell^2\)- and \(\ell^q\)-analogues of \cite[Thm.~3.1]{khesin_GAFA} and \cite[Thm.~4.10]{bauerLpFisherRaoMetricAmari2024}. In the \(\ell^q\)-setting, we also discuss the corresponding Amari--\v{C}encov \(\alpha\)-connection and the existence of a Chern connection.

Our \textbf{second main contribution} is an application of this \(\ell^2\)-information geometry to infinite-dimensional optimization and Hamiltonian dynamics. In \Cref{s: 3}, we study an infinite-dimensional linear programming problem on the \(\ell^2\)-probability simplex and solve it via the gradient flow of the \(\ell^2\)-Fisher--Rao metric. We then show that these gradient flow lines are precisely geodesics of an exponential connection and that the resulting geodesic flow is complete. In \Cref{S: 4}, we further show that this gradient flow admits a Hamiltonian interpretation on an infinite-dimensional Kähler manifold and gives rise to infinitely many Poisson-commuting first integrals.\\

The present article is an invited extended version of \cite{Maier_GSI_2026}. Relative to the conference paper, it develops the \(\ell^q\)-information-geometric perspective further, adds the \(e\)-geodesic flow together with a Hamiltonian interpretation in the infinite-dimensional setting, and outlines a possible extension to non-compact manifolds. The Hamiltonian viewpoint pursued here is motivated by the information-geometric approach to linear programming introduced in \cite{faybusovich_hamiltonian_1991}. At the same time, the construction suggests that the \(\ell^2\)-probability simplex, equipped with the pullback differentiable structure introduced here, may serve as a useful model for the Fisher--Rao geometry of spaces of probability densities on non-compact Riemannian manifolds.\\

\noindent\textbf{Organization of the article.}
In \Cref{s: 2} we develop the \(\ell^2\)- and \(\ell^q\)-information-geometric frameworks on the infinite-dimensional probability simplex, prove that the square-root and \(q\)-root transforms are isometries, and in the \(\ell^q\)-case discuss the corresponding Amari--\v{C}encov \(\alpha\)-connection and a Chern connection. In \Cref{s: 3} we study an infinite-dimensional linear optimization problem on the \(\ell^2\)-probability simplex, derive the associated Fisher--Rao gradient flow, and reinterpret it as the geodesic flow of an exponential connection, proving geodesic completeness. In \Cref{S: 4} we show that these dynamics admit a Hamiltonian interpretation on infinite-dimensional complex projective space and are accompanied by infinitely many Poisson-commuting first integrals. Finally, in \Cref{S:5} we outline a corresponding Fisher--Rao geometry on \(\mathrm{Dens}(M)\) for possibly non-compact Riemannian manifolds.
\section{Introducing $\ell^2$ and $\ell^q$-Information Geometry}\label{s: 2}
We begin with an overview of this section. In \Cref{ss:ell_2_info_geo}, we introduce \(\ell^2\)-information geometry: we equip the probability simplex with a non-canonical differentiable structure that allows us to define the \(\ell^2\)-analogue of the Fisher--Rao metric in this framework, and, as a proof of concept, we show that the probability simplex equipped with this metric is isometric to an open subset of the \(\ell^2\)-sphere.\\
Next, in \Cref{ss:l_q_information_geometry}, we introduce the \(\ell^q\)-analogue of \Cref{ss:ell_2_info_geo}. In addition to the results obtained in the \(\ell^2\)-framework, we construct the Amari--\v{C}encov \(\alpha\)-connection and a Chern connection for the resulting \(\ell^q\)-Fisher--Rao metric, which is a strong Finsler metric.
\subsection{$\ell^2$-Information Geometry.}\label{ss:ell_2_info_geo} We begin by introducing the setting in detail. We denote by $
\ell^2_\mathbb{R}$ the space of real-valued sequences $(x_n)_{n\in \mathbb{N}} \subseteq \mathbb{R}$ satisfying $$\sum_{n=0}^{\infty} x_n^2 < \infty.$$
For convenience, we will later denote such sequences simply by $(x_n)$. This space is a real Hilbert space if equipped  with the inner product \[
\langle (x_n), (y_n)\rangle_{\ell^2} = \sum_{n=0}^{\infty} x_n y_n.\]
The unit sphere in $(\ell^2_{\mathbb{R}},\langle \cdot,\cdot \rangle_\RR$ is 
\[
\mathbb{S}\left(\ell^2_{\mathbb{R}} \right)=\left\{(x_n)\in \ell_{\mathbb{R}}^2 : \sum_{n=0}^{\infty} x_n^2 = 1 \right\},
\]
with the tangent space at the point $(x_n) \in \mathbb{S}(\ell^2_{\mathbb{R}})$ given by
\[
T_{(x_n)} \mathbb{S}(\ell^2_{\mathbb{R}}) = \left\{(v_n) \in \ell_{\mathbb{R}}^2 : \langle (x_n), (v_n) \rangle_{\ell^2_{\mathbb{R}}} = 0 \right\}.
\]
In this setting, the round metric $\Glt$, respectively the $\ell^2$-metric, is precisely the restriction of $\langle\cdot, \cdot \rangle_{\ell^2}$ onto $\mathbb{S}(\ell^2_{\mathbb{R}})$. It is well known that \( (\mathbb{S}(\ell^2_{\mathbb{R}}), \Glt) \) is a strong Riemannian Hilbert manifold in the sense of \cite{La99}, that is \[\Glt:T(\mathbb{S}(\ell^2_{\mathbb{R}})\longrightarrow T^*(\mathbb{S}(\ell^2_{\mathbb{R}})\] is a bundle isomorphism.
The open subset of strictly positive sequences is denoted by 
\begin{equation}\label{e: definition Ul2 }
	\U:=\left\{(x_n)\in \Slr:\  x_n>0 \quad \forall n\in \NN\right\},
\end{equation}
is as an open subset of the strong Hilbert manifold $(\mathbb{S}(\ell^2_{\mathbb{R}}), \Glt)$ equipped with the round metric $\Glt$ also a strong Riemannian Hilbert manifold.\\ 
We move on by introducing the $\ell^2_{\mathbb{R}}$-analogue of the space of probability densities, the $\ell^2_{\mathbb{R}}$-\emph{probability simplex}, defined by 
\begin{equation}\label{e: definition l2probability simplex}
	\Delta := \left\{(p_n) \in \ell^1_{\mathbb{R}} : \sum_{n=0}^{\infty} p_n = 1 \quad \text{and} \quad p_n > 0 \quad \forall n \in \mathbb{N} \right\}.
\end{equation}
Note, that at this point, we have several choices for how to equip $\Delta$ with a differentiable structure, see \Cref{r: the right diff strcuture}. We choose the differentiable structure on $\Delta$ so that the following homeomorphism, called the \emph{square-root map}, becomes a diffeomorphism:
\[
	\varPhi: \Delta \longrightarrow \mathcal{U}, \quad (p_n) \mapsto (\sqrt{p_n}).
\]
This structure is completely different from the differentiable structure induced by the ambient space $\ell^2_{\mathbb{R}}$. The tangent space to $\Delta$ at a point $(p_n) \in \Delta$, with respect to this differentiable structure, is given by
\begin{equation}
\label{d:differentiable_structure}
    T_{(p_n)}\Delta = \left\{ (v_n) \in \ell^1_{\mathbb{R}} : \sum_{n=0}^{\infty} v_n = 0 \quad \text{and} \quad \left(\frac{v_n}{\sqrt{p_n}}\right) \in \ell^2_{\mathbb{R}} \right\}.
\end{equation}
By introducing the $\ell^2$-\emph{Fisher--Rao} metric as:
\begin{equation}\label{e:defi_l2_Fisher_Rao}
	\mathcal{G}^{\mathrm{FR}}_{(p_n)}\left((v_n), (w_n) \right) := \frac{1}{4}\sum_{n=0}^{\infty} \frac{v_n \cdot w_n}{p_n}, \quad \forall (v_n), (w_n) \in T_{(p_n)}\Delta,
\end{equation}
the space $(\Delta, \mathcal{G}^{\mathrm{FR}})$ becomes a strong Riemannian Hilbert manifold. Before we move on, we note the following:

\begin{remark}\label{r: the right diff strcuture}
A condition similar to the one on \( \left(v_n\,/\sqrt{p_n}\,\right) \) in \eqref{d:differentiable_structure} also appears in \cite[§2]{Friedrich1991}. Without this condition, the \( \ell^2 \)-Fisher–Rao metric \eqref{e:defi_l2_Fisher_Rao} is not well-defined.\\
Even under decay assumptions analogous to those in \cite[§3.5]{MichorMumfordZoo}, one can construct rapidly decaying sequences in \( \Delta \) such that the \( \ell^2 \)-Fisher–Rao metric fails to be finite on some tangent vectors. 
\end{remark}
Subject to our chosen differentiable structure in \eqref{d:differentiable_structure}, and by a computation similar to \cite[Thm. 3.1]{khesin_GAFA}, it follows directly that:

\begin{theorem}\label{t: square root map is isometry}
The square-root map \( \varPhi \) defined by
\[
\varPhi: (\Delta, \mathcal{G}^{\mathrm{FR}}) \longrightarrow (\mathcal{U}, \mathcal{G}^{\mathrm{L}^2}), \quad (p_n) \mapsto (\sqrt{p_n}\,),
\]
is an isometry.
\end{theorem}

\begin{remark}
This result can be seen as a blend of \cite[Thm.~3.1]{khesin_GAFA} and the methods in \cite{Friedrich1991}, but does not follow immediately from them. Accordingly, \Cref{t: square root map is isometry} provides yet another infinite-dimensional analogue of \cite[Proposition~2.1]{InfGeo_book}, and it relies crucially on the differentiable structure chosen in \eqref{d:differentiable_structure}, as emphasized in \Cref{r: the right diff strcuture}.\\
Interestingly, finite-dimensional probability simplices $\Delta^N $ cannot be embedded as totally geodesic submanifolds into the infinite-dimensional simplex \( \Delta \). Indeed, \( \Delta^N \subseteq \partial \Delta \), which illustrates how \( \ell^2_{\mathbb{R}} \)-information geometry differs fundamentally from its finite-dimensional counterpart.\\
Moreover, an element of \( \Delta \) can be interpreted as a discretization of a probability measure subject to infinitely many measurements.
Note that \Cref{t: square root map is isometry} is effectively illustrated by the \Cref{fig:sqrt-map}.
\end{remark}
\begin{figure}[ht]
	\centering
\includegraphics{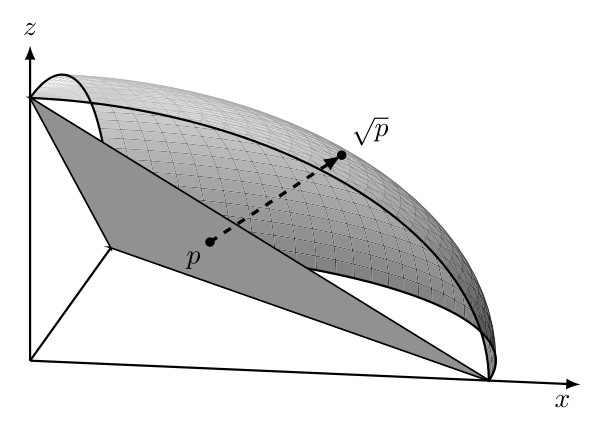}
	\caption{Illustration of the square-root map as an isometry.}
	\label{fig:sqrt-map}
\end{figure}
A natural question that arises for readers familiar with infinite-dimensional information geometry is the following:  
What is the analogue of \cite{FisherRaoisunique} in the setting of $\ell^2$-information geometry?

As a first illustration of \Cref{t: square root map is isometry}, we observe that $(\Delta, \Gfr)$ is geodesically convex; that is, for every pair of points, there exists a length-minimizing geodesic connecting them. This follows directly from the well-known fact that $(\Slr, \Glt)$ is geodesically convex, and consequently, $(\U, \Glt)$ is also geodesically convex.  
Therefore, by \Cref{t: square root map is isometry}, we obtain:
\begin{corollary}
	The $\ell^2$-probability simplex $\Delta$, equipped with the $\ell^2$-Fisher–Rao metric $\Gfr$, is geodesically convex. 
\end{corollary}
\begin{remark}
	At this point, we note that the Hopf–Rinow theorem is famously wrong in the context of infinite-dimensional Riemannian Hilbert manifolds \cite{HopfrinowfalseAktkin, HopfrinowfalseEkeland}.  
	Moreover, geodesic completeness does not necessarily imply geodesic convexity. Interestingly, the latter does hold for half-Lie groups equipped with a strong right-invariant Riemannian metric \cite[Thm. 7.7.]{Bauer_2025}. For Hamiltonian versions of the Hopf-Rinow theorem in this setting we refer to \cite{HopfRinowHalfLiegroups, MaierRuscelliTonelli}.
\end{remark}
We conclude this section by presenting  in the next subsection the $\ell^q_{\RR}$-analogue of the recent and intriguing work of Bauer–Le Brigant–Lu–Maor~\cite{bauerLpFisherRaoMetricAmari2024}.

\subsection{\(\ell^q\)-Information Geometry}\label{ss:l_q_information_geometry}

By \(\ell^q_{\RR}\) we denote the space of all real sequences \((x_n)\) such that the \(\ell^q_{\RR}\)-norm \(\Vert \cdot \Vert_{\ell^q}\) is finite, i.e.,
\[
\Vert (x_n) \Vert_{\ell^q} := \left( \sum_{n=0}^{\infty} |x_n|^q \right)^{1/q} < \infty,
\]
where \(q \in (1, \infty)\). We adopt the notation \(q\) instead of \(p\), as in~\cite{bauerLpFisherRaoMetricAmari2024}, in order to avoid confusion with elements \((p_n) \in \Delta\).

We define \(\mathcal{U}_q\) as the following open subset of the \(\ell^q_{\RR}\)-unit sphere \(\mathbb{S}(\ell^q_{\RR})\):
\[
\mathcal{U}_q := \left\{ (x_n) \in \mathbb{S}(\ell^q_{\RR}) : x_n > 0 \quad \forall n \in \mathbb{N} \right\}.
\]
We equip \(\Delta\) with a differentiable structure such that the homeomorphism, called the \emph{\(q\)-root transform}, defined by
\begin{equation}\label{e: def q-root transform}
    \varPhi_q: \Delta \longrightarrow \mathcal{U}_q, \quad (p_n) \mapsto \left(p_n^{1/q}\right),
\end{equation}
becomes a diffeomorphism. For \(q \neq 2\), we denote by \(\Delta^q\) the set \(\Delta\) equipped with this differentiable structure. The tangent space at a point \((p_n) \in \Delta^q\) is then given by
\begin{equation*}\label{e: q diff structure on delta}
	T_{(p_n)}\Delta^q := \left\{ (v_n) \in \ell^1_{\mathbb{R}} : \sum_{n=0}^{\infty} v_n = 0 \quad \text{and} \quad \left(v_n p_n^{(1/q - 1)}\right) \in \ell^q_{\mathbb{R}} \right\}.
\end{equation*}

We next introduce the \(\ell^q\)-Fisher--Rao metric, which defines a Finsler metric for each \(q \in (1, \infty)\), by
\begin{equation}\label{e: def q-Fisher-Rao metric}
    \mathcal{F}^q_{(p_n)}\left( (v_n) \right) := \left( \sum_{n=0}^{\infty} \left\vert \frac{v_n}{p_n} \right\vert^q \cdot p_n \right)^{\frac{1}{q}},
    \quad \forall\ ((p_n), (v_n)) \in T\Delta^q.
\end{equation}
As in the \(q\)-root framework developed in~\cite{bauerLpFisherRaoMetricAmari2024}, the key point is that the \(q\)-root transform identifies this Finsler structure with the ambient \(\ell^q\)-geometry. In our setting, the same mechanism carries over from the finite-dimensional and \(L^q\)-contexts to the \(\ell^q\)-setting. More precisely, one obtains the following.

\begin{theorem}\label{T: q-root transform as isom}
    The \(q\)-root transform \(\varPhi_q\) is an isometry between \(\left( \Delta^q, \mathcal{F}^q \right)\) and \(\left( \mathcal{U}_q, \Vert \cdot \Vert_{\ell^q} \right)\).
\end{theorem}

In \cite{bauerLpFisherRaoMetricAmari2024}, the construction of the Amari--\v{C}encov \(\alpha\)-connections in~\cite[Lemma~4.1]{bauerLpFisherRaoMetricAmari2024} and the Chern connection in~\cite[Thm.~4.6]{bauerLpFisherRaoMetricAmari2024} is driven by the fact that the \(q\)-root transform is an isometry. In the following, we indicate the corresponding constructions for \((\Delta^q, \mathcal{F}^q)\). The point is not to develop the full theory here, but rather to record that the basic constructions extend naturally to the present \(\ell^q\)-framework.

\begin{corollary}\label{L: alpha-connection on Deltaq}
For \(q \in (1,\infty)\), set \(\alpha := 1-\frac{2}{q}\) and \(q^* := \frac{q}{q-1}\).
Then, for each \((p_n) \in \Delta^q\), the map defined on tangent vectors \((v_n),(w_n)\in T_{(p_n)}\Delta^q\) by
\begin{equation}\label{e: alpha-connection on Deltaq}
    \overline{\nabla}^{(\alpha)}_{(v_n)} (w_n)
    =
    D_{(v_n)} (w_n)
    - \frac{1}{q^*}
    \left(
        \frac{v_n}{p_n}\, w_n
        -
        \Bigl(
            \sum_{n=0}^\infty \frac{v_n w_n}{p_n}
        \Bigr) p_n
    \right)
\end{equation}
defines the Amari--\v{C}encov \(\alpha\)-connection on \((\Delta^q, \mathcal{F}^q)\).
\end{corollary}

\begin{proof}
The construction is the natural \(\ell^q\)-analogue of the argument in \cite[Lemma~4.1]{bauerLpFisherRaoMetricAmari2024}. Using \Cref{T: q-root transform as isom}, the same computation carries over to the present sequence setting, with sums replacing integrals.
\end{proof}

We close this subsection by recording the corresponding existence result for the Chern connection on the Finsler manifold \((\Delta^q,\mathcal F^q)\). We note that, unlike in the case of strong Riemannian manifolds, a Chern connection need not exist for infinite-dimensional Finsler manifolds, even if the Finsler structure is strong. In the present setting, however, the \(q\)-root transform allows one to transfer the relevant construction from the ambient \(\ell^q\)-geometry.

\begin{corollary}
    \label{T: existence chern connection on Deltaq}
    The Finsler manifold \((\Delta^q,\mathcal F^q)\) admits a Chern connection.
\end{corollary}

\begin{proof}
Using \Cref{T: q-root transform as isom}, one obtains the Chern connection on \((\Delta^q,\mathcal F^q)\) by the same construction as in \cite[Thm.~4.6]{bauerLpFisherRaoMetricAmari2024}, adapted to the present \(\ell^q\)-setting.
\end{proof}
\section{$\ell^2$-Information geometry of an optimization problem}\label{s: 3}
The aim of this section is to use the previously developed geometric framework of \(\ell^2\)-information geometry to extend the finite-dimensional information-geometric approach of \cite{faybusovich_hamiltonian_1991}, originally developed for finite-dimensional linear programming problems, to the infinite-dimensional setting. First, in \Cref{ss:Solutions of optimization problem via a gradient flow}, we achieve this by means of a gradient flow in that framework. Later, in \Cref{ss:The e-geodesic flow}, we interpret these gradient flows as geodesic flows of a certain affine connection. Along the way, we prove that this geodesic flow is geodesically complete, i.e., that geodesics exist for all time.

\subsection{Solutions of the optimization problem via a gradient flow}\label{ss:Solutions of optimization problem via a gradient flow}
We consider the following linear programming problem on the closure \(\bar{\Delta}\) of the \(\ell^2\)-probability simplex:
\begin{equation}\tag{LP}\label{eq:ell2_LPP}
	\max_{(p_n)\in \bar{\Delta}} \langle (c_n), (p_n) \rangle_{\ell^2_\mathbb{R}},
	\qquad (c_n)\in \ell^2_\mathbb{R}.
\end{equation}
Defining the smooth function
\begin{equation}\label{eq:def_F_cn}
F_{(c_n)}:\Delta\longrightarrow\mathbb{R},
\qquad (p_n)\longmapsto \langle (c_n), (p_n) \rangle_{\ell^2},
\end{equation}
the problem \eqref{eq:ell2_LPP} can be interpreted as the problem of maximizing this function over \(\bar{\Delta}\). We now describe the gradient flow lines of \(F_{(c_n)}\) with respect to \((\Delta,\mathcal{G}^{\mathrm{FR}})\) and show that they converge exponentially fast to solutions of \eqref{eq:ell2_LPP}.

\begin{proposition}
\label{prop:integral-curves}
Let \(t \mapsto (p_n)(t)\) be a gradient flow line of \(F_{(c_n)}\), defined in \eqref{eq:def_F_cn}, on \((\Delta,\mathcal{G}^{\mathrm{FR}})\) with initial value \((p_n)(0) = (p_n)_0 \in \Delta\). Then, for all \(n \in \mathbb{N}\) and all \(t \in [0,\infty)\), the components of \((p_n)(t)\) are given by
\begin{equation}\label{eq:integral-curves}
	p_n(t)=\frac{p_n(0)e^{c_n t}}{\sum_{k=0}^{\infty} p_k(0)e^{c_k t}},\quad \forall t\in [0,\infty).
\end{equation}
If, in addition, the sequence \((c_n)\) in \eqref{eq:ell2_LPP} is strictly monotonically decreasing, then the limit
\[
p_{\max}:=\lim_{t\to\infty}(p_n)(t)
\]
exists and is a solution of \eqref{eq:ell2_LPP}.
\end{proposition}

\begin{proof}
By applying \Cref{t: square root map is isometry} and noting that the square root map extends to a homeomorphism from the closure \( \bar{\Delta} \) onto \( \bar{\U} \), we find that \eqref{eq:ell2_LPP} is equivalent to the following optimization problem over the closure \( \bar{\U} \) of \( \U \) in \( \Slr \):
\begin{equation}\tag{NLP}\label{eq: optimization problem sphere}
	\max_{(x_n) \in \bar{\U}} \langle (c_n), (x_n^2) \rangle_{\ell^2_\R}, \quad (c_n)\in \ell^2_\R.
\end{equation}

To analyze this, we define the smooth function
\[
H_{(c_n)}: \ell^2(\RR)\longrightarrow \RR: (x_n)\mapsto \langle (c_n), (x_n^2) \rangle_{\ell^2},
\]
where \( (c_n) \in \ell^2_{\mathbb{R}} \) is fixed. The problem \eqref{eq: optimization problem sphere} then reduces to maximizing \( H_{(c_n)} \) over \( \bar{\U} \). The gradient of \( H_{(c_n)} \) in the ambient Hilbert space \( \ell^2(\RR) \) is given by
\[
\nabla H_{(c_n)}(x_n) = 2\,\mathrm{Diag}(x_n)(c_n) := 2(x_n c_n).
\]
Projecting this vector onto the tangent space \( T_{(x_n)}\U \) using the orthogonal projection
\[
P_{(x_n)}(v_n) = (v_n) - \langle (v_n), (x_n) \rangle_{\ell^2} (x_n),
\]
yields the Riemannian gradient of \( H_{(c_n)} \) on the sphere \( (\U, \Glt) \):
\begin{equation}\label{eq: gradient on sphere}
	\mathrm{grad}^{\Glt}(H_{(c_n)})(x_n) = P_{(x_n)}\left( 2\,\mathrm{Diag}(x_n)(c_n) \right), \quad \forall (x_n)\in\U.
\end{equation}

By invoking \Cref{t: square root map is isometry} once more, we obtain from \eqref{eq: gradient on sphere} that the Riemannian gradient of the function \(F_{(c_n)}\), defined in \eqref{eq:def_F_cn}, on \( (\Delta, \Gfr) \) is given by
\begin{equation}\label{def:W_c}
\mathrm{grad}^{\Gfr}(F_{(c_n)})(p_n) = \operatorname{Diag}(p_n)(c_n) - \langle (p_n), (c_n) \rangle_{\ell^2_\R} (p_n), \quad \forall (p_n) \in \Delta.
\end{equation}
We denote \(\mathrm{grad}^{\Gfr}(F_{(c_n)})\) by \( W_{(c_n)} \) for convenience. It is standard to verify that \( W_{(c_n)} \) satisfies a global Lipschitz condition. Hence, for any \( (c_n) \in \ell^2_\R \) and initial condition \( (p_n)_0 \in \Delta \), there exists a unique solution to the initial value problem
\begin{equation}\label{eq:IVP}
	\frac{\d}{\d t}(p_n)(t) = W_{(c_n)}(p_n), \quad (p_n)(0) = (p_n)_0 \in \Delta,
\end{equation}
i.e., the gradient flow of \( F_{(c_n)} \) is well defined on \( \Delta \).

Another straightforward computation confirms that the curves \( (p_n)(t) \) given in \eqref{eq:integral-curves} are indeed solutions to \eqref{eq:IVP}.

Furthermore, if the sequence \( (c_n) \) in \eqref{eq:ell2_LPP} is strictly monotonically decreasing, then the solution converges to the point
\[
p_{\max} := \lim_{t \to \infty} (p_n)(t) = (1, 0, \dots) \in \bar{\Delta},
\]
which is clearly a solution of \eqref{eq:ell2_LPP}.
\end{proof}

In finite-dimensional information geometry, the analogues of the gradient flow lines in \Cref{prop:integral-curves} are known as \(e\)-geodesics, i.e., geodesics with respect to a certain affine connection. In the next subsection, we extend this finite-dimensional observation to the setting of \(\ell^2\)-information geometry.
\subsection{The \(e\)-geodesic flow in \(\ell^2\)-information geometry and the optimization problem~\eqref{eq:ell2_LPP}}\label{ss:The e-geodesic flow}
We begin by defining the \(e\)-connection.

\begin{definition}[Exponential connection]
The \emph{\(e\)-connection}
\[
\nabla^{(e)}:\mathfrak{X}(\Delta)\times\mathfrak{X}(\Delta)\longrightarrow \mathfrak{X}(\Delta), 
\qquad ((V_n),(W_n))\longmapsto \nabla^{(e)}_{(V_n)}(W_n),
\]
is defined pointwise at each \((p_n)\in \Delta\), where the \(n\)-th entry of the sequence
\[
\restr{\nabla^{(e)}_{(V_n)}(W_n)}{(p_n)}
:=
\bigl(\nabla^{(e)}_{(v_n)}(w_n)\bigr)_n
\]
is given by
\begin{equation}\label{eq:def_e_con_comp}
\bigl(\nabla^{(e)}_{(v_n)}(w_n)\bigr)_n
:=
p_n
\restr{\left(
D_{(V_n)}\!\left(\frac{W_n}{p_n}\right)
-
\sum_{k=0}^{\infty}
p_k\,D_{(V_n)}\!\left(\frac{W_k}{p_k}\right)
\right)}{(p_n)},
\end{equation}
where \(D_{(V_n)}\) denotes the directional derivative along \((V_n)\).
\end{definition}

Next, we verify that the \(e\)-connection \(\nabla^{(e)}\) is indeed an affine connection.

\begin{proposition}[Exponential connection is an affine connection]
The operator \(\nabla^{(e)}\) defines an affine connection on \(T\Delta\).
\end{proposition}

\begin{proof}
Using the definition of the tangent space \(T_{(p_n)}\Delta\) in \eqref{d:differentiable_structure}, it is straightforward to check that \(\nabla^{(e)}\) is well defined, namely that
\[
\sum_{n=0}^{\infty}\bigl(\nabla^{(e)}_{(v_n)}(w_n)\bigr)_n = 0.
\]

It remains to verify that \(\nabla^{(e)}\) satisfies the axioms of an affine connection.

\textit{Linearity in the first argument.}
By the linearity of the directional derivative and \eqref{eq:def_e_con_comp}, we have
\[
\nabla^{(e)}_{f(U_n)+g(V_n)}(W_n)
=
f\nabla^{(e)}_{(U_n)}(W_n)
+
g\nabla^{(e)}_{(V_n)}(W_n)
\]
for all \(f,g\in C^{\infty}(\Delta)\) and all \((U_n),(V_n),(W_n)\in \mathfrak{X}(\Delta)\).

\textit{Linearity in the second argument.}
This follows directly from the definition of \(\nabla^{(e)}\) in \eqref{eq:def_e_con_comp}.

\textit{Leibniz rule.}
Let \(f\in C^\infty(\Delta)\). By \eqref{eq:def_e_con_comp}, together with the Leibniz rule for directional derivatives, we compute
\begin{align*}
\bigl(\nabla^{(e)}_{(V_n)}(f(W_n))\bigr)_n
&=
p_n\restr{\left(
D_{(V_n)}(f)\frac{W_n}{p_n}
+
f\,D_{(V_n)}\!\left(\frac{W_n}{p_n}\right)\right)}{(p_n)}\\
&\quad
- p_n \restr{\left(
\sum_{k=0}^{\infty}p_k\,D_{(V_n)}(f)\frac{W_k}{p_k}
+
\sum_{k=0}^{\infty}p_k\,f\,D_{(V_n)}\!\left(\frac{W_k}{p_k}\right)
\right)}{(p_n)}.
\end{align*}
Since \(D_{(V_n)}(f)\) does not depend on the summation index \(k\), and using the identity for tangent vectors from \eqref{d:differentiable_structure}, we obtain
\begin{align*}
\bigl(\nabla^{(e)}_{(V_n)}(f(W_n))\bigr)_n
&=
p_n\restr{\left(
D_{(V_n)}(f)\frac{W_n}{p_n}
+
f\,D_{(V_n)}\!\left(\frac{W_n}{p_n}\right)\right)}{(p_n)}\\
&\quad
- p_n\restr{\left(
D_{(V_n)}(f)\sum_{k=0}^{\infty}W_k
+
f\sum_{k=0}^{\infty}p_k\,D_{(V_n)}\!\left(\frac{W_k}{p_k}\right)
\right)}{(p_n)} \\
&=
\bigl(D_{(V_n)}(f)\,W_n\bigr)_n
+
f\bigl(\nabla^{(e)}_{(V_n)}(W_n)\bigr)_n \quad \forall n\in \NN.
\end{align*}
Hence
\[
\nabla^{(e)}_{(V_n)}\bigl(f(W_n)\bigr)
=
D_{(V_n)}(f)\,(W_n)
+
f\,\nabla^{(e)}_{(V_n)}(W_n)
\quad \forall f\in C^{\infty}(\Delta),\ \forall (V_n),(W_n)\in \mathfrak{X}(\Delta).
\]

Thus \(\nabla^{(e)}\) satisfies all axioms of an affine connection.
\end{proof}

We now give the promised interpretation that the gradient flow lines solving \eqref{eq:ell2_LPP} are precisely the \(e\)-geodesics. To this end, we introduce the following definition.

\begin{definition}[\(e\)-geodesics]\label{def:e-geodesic}
A smooth curve \(t\mapsto (p_n)(t)\in C^{\infty}(I,\Delta)\) is called an \emph{\(e\)-geodesic} of \((\Delta,\nabla^{(e)})\) if
\[
\nabla^{(e)}_{(\dot p_n)}(\dot p_n)=0.
\]
\end{definition}

Since the infinite-dimensional manifold \(\Delta\) is not compact, it is a priori unclear whether \((\Delta,\nabla^{(e)})\) is geodesically complete, that is, whether \(e\)-geodesics exist globally in time. In the next theorem, we prove that \((\Delta,\nabla^{(e)})\) is geodesically complete.

To establish the promised relation between \eqref{eq:ell2_LPP} and \(e\)-geodesics, we derive the following useful characterization.

\begin{theorem}[Geodesic completeness]\label{prop:e_geodesics_equation}
Let \(t\mapsto (p_n)(t)\) be a smooth curve in \(\Delta\) with initial data
\[
(p_n)(0)=(p_n)_0\in \Delta,
\qquad
(\dot p_n)(0)=(\dot p_n)_0\in T_{(p_n)_0}\Delta.
\]
Then \(t\mapsto (p_n)(t)\) is an \(e\)-geodesic of \((\Delta,\nabla^{(e)})\) if and only if there exists \(\lambda\in\mathbb{R}\) such that, for the sequence \((a_n)\) defined by
\[
a_n:=\frac{\dot p_n(0)}{p_n(0)}+\lambda
\qquad \forall n\in\mathbb{N},
\]
the curve is of the form
\begin{equation}\label{eq:explicit_form_e_geodesics}
p_n(t)=\frac{p_n(0)e^{a_n t}}{\sum_{k=0}^{\infty} p_k(0)e^{a_k t}}
\qquad \forall n\in\mathbb{N},\ \forall t\in\mathbb{R}.
\end{equation}
Consequently, \((\Delta,\nabla^{(e)})\) is geodesically complete.
\end{theorem}

\begin{remark}
The sequence \((a_n)\) is uniquely determined by the initial data only up to addition of a common constant. Indeed, replacing \(a_n\) by \(a_n+\mu\) leaves \eqref{eq:explicit_form_e_geodesics} unchanged, since the factor \(e^{\mu t}\) cancels between numerator and denominator.
\end{remark}

\begin{proof}
By the definition of the \(e\)-connection,
\[
\bigl(\nabla^{(e)}_{\dot p}\dot p\bigr)_n
=
p_n
\left(
\frac{\mathrm d}{\mathrm dt}\left(\frac{\dot p_n}{p_n}\right)
-
\sum_{k=0}^{\infty} p_k\frac{\mathrm d}{\mathrm dt}\left(\frac{\dot p_k}{p_k}\right)
\right)
\qquad \forall n\in \mathbb{N}.
\]
Hence \(\nabla^{(e)}_{\dot p}\dot p=0\) if and only if
\begin{equation}\label{eq:d_dt_dot_p_p}
\frac{\mathrm d}{\mathrm dt}\left(\frac{\dot p_n}{p_n}\right)
=
\sum_{k=0}^{\infty} p_k\frac{\mathrm d}{\mathrm dt}\left(\frac{\dot p_k}{p_k}\right)
\qquad \forall n\in\mathbb{N}.
\end{equation}
Since the right-hand side of \eqref{eq:d_dt_dot_p_p} is independent of \(n\), there exist constants \(a_n\) and a function \(c(t)\) such that
\begin{equation}\label{eq:dot_p_n_p_n}
\frac{\dot p_n(t)}{p_n(t)}=a_n-c(t).
\end{equation}
By the definition of the tangent space of \(\Delta\) in~\eqref{d:differentiable_structure}, summing \eqref{eq:dot_p_n_p_n} over \(n\) yields
\begin{equation}\label{eq:c_t}
c(t)=\sum_{k=0}^{\infty}p_k(t)a_k.
\end{equation}
Substituting \eqref{eq:c_t} into \eqref{eq:dot_p_n_p_n}, we obtain
\begin{equation}\label{eq:dot_pn_p_n_a_n}
\frac{\dot p_n(t)}{p_n(t)}
=
a_n-\sum_{k=0}^{\infty}p_k(t)a_k.
\end{equation}
Evaluating \eqref{eq:dot_pn_p_n_a_n} at \(t=0\), we further obtain
\[
\frac{\dot p_n(0)}{p_n(0)}
=
a_n-\sum_{k=0}^{\infty}p_k(0)a_k.
\]
Thus there exists \(\lambda\in\mathbb{R}\), namely \(\lambda:=\sum_{k=0}^{\infty}p_k(0)a_k\), such that
\[
a_n=\frac{\dot p_n(0)}{p_n(0)}+\lambda
\qquad \forall n\in\mathbb{N}.
\]
Finally, solving \eqref{eq:dot_pn_p_n_a_n} yields
\[
p_n(t)
=
\frac{p_n(0)e^{a_nt}}
{\sum_{k=0}^{\infty}p_k(0)e^{a_kt}} \quad \forall t\in \RR.
\]
The converse implication is proved analogously.
\end{proof}

We close this subsection by establishing the relation between the gradient flow lines in \Cref{prop:integral-curves}, converging to solutions of \eqref{eq:ell2_LPP}, and \(e\)-geodesics.

\begin{corollary}
Let \((c_n)\in \ell^2_\RR\), let \((p_n)_0\in \Delta\), and let \(t\mapsto (p_n)(t)\) be a smooth curve in \(\Delta\) with initial value
\[
(p_n)(0)=(p_n)_0.
\]
Then \(t\mapsto (p_n)(t)\) is the gradient flow line of \(F_{(c_n)}\), defined in \eqref{eq:def_F_cn}, with initial value \((p_n)_0\) if and only if it is an \(e\)-geodesic of \((\Delta,\nabla^{(e)})\) with initial values \(((p_n)_0,(v_n))\) for some \((v_n)\in T_{(p_n)_0}\Delta\) such that there exists \(\lambda\in \RR\) satisfying
\[
c_n = \frac{v_n}{(p_n)_0} + \lambda
\qquad \forall n\in \NN.
\]
\end{corollary}

\begin{remark}
This settles the author's conjecture stated at the end of \cite[§3]{Maier_GSI_2026}.
\end{remark}

\begin{proof}
    First observe that, for each $(c_n)\in \ell^2_\RR$, there exist $(p_n)(0)=(p_n)_0\in \Delta$, $(v_n)\in T_{(p_n)_0}\Delta$, and $\lambda\in \RR$ such that
    \[
    c_n = \frac{v_n}{p_n(0)} + \lambda \qquad \forall n\in \NN.
    \]
    The claim then follows directly from \Cref{prop:integral-curves} and \Cref{prop:e_geodesics_equation}.
\end{proof}

\section{An infinite-dimensional integrable Hamiltonian system}\label{S: 4}
The aim of this section is to study the Hamiltonian nature of the gradient flows in \Cref{s: 3} and the underlying symmetries.

To this end, we denote by \(\ell_{\mathbb{C}}^2\) the space of complex-valued sequences \((z_n)\subseteq \mathbb{C}\) such that
\[
\| (z_n)\|_{\ell^2}^2:=\sum_{n=0}^{\infty} |z_n|^2 < \infty.
\]
Equipped with the standard Hermitian \(\ell^2\)-inner product \(\langle \cdot, \cdot \rangle_{\ell^2}\), this space carries the structure of an infinite-dimensional Kähler manifold, since
\[
\Re \langle \mathrm{i}\cdot, \cdot \rangle_{\ell^2} = \Im \langle \cdot, \cdot \rangle_{\ell^2},
\]
where \(\mathrm{i}=\sqrt{-1}\). The action \(\mathbb{S}^1 \curvearrowright \ell_{\mathbb{C}}^2\), where \(\S^1=\{e^{\i t}: t\in \RR\}\), given by
\[
e^{it} \cdot (z_n) := (e^{it}z_n),
\]
acts by Kähler morphisms and is Hamiltonian, with momentum map
\begin{equation} \label{eq:S1_momentmap_cpinf}
	\mu_{\mathbb{S}^1} : \ell^2_{\mathbb{C}} \longrightarrow  \mathrm{i} \mathbb{R}, \quad (z_n) \mapsto \mathrm{i} \langle (z_n), (z_n) \rangle_{\ell^2}.
\end{equation}
Since \(\mathrm{i}\) is a regular value of \(\mu_{\mathbb{S}^1}\), we obtain by Kähler reduction that the quotient
\[
\mu_{\mathbb{S}^1}^{-1}(\mathrm{i}) / \mathbb{S}^1
\]
is a Kähler manifold. As \(\mu_{\mathbb{S}^1}^{-1}(\mathrm{i})\) is precisely the unit sphere \(\mathbb{S}(\ell_{\mathbb{C}}^2)\) in \(\ell_{\mathbb{C}}^2\), we obtain
\[
\mu_{\mathbb{S}^1}^{-1}(\mathrm{i}) / \mathbb{S}^1 = \mathbb{CP}^{\infty},
\]
where the induced Kähler structure is precisely the Fubini--Study metric \(\mathcal{G}^{\mathrm{FS}}\) together with the Fubini--Study form \(\Omega^{\mathrm{FS}}\). We refer, for example, to \cite[§2]{khesin_geometry_2019} for the explicit form of this Kähler structure. We identify the infinite-dimensional torus
\[
\mathbb{T}^{\infty} := \prod_{n=1}^{\infty} \mathbb{S}^1
\]
with a subgroup of the diagonal unitary operators on \(\ell^2_{\CC}\) by
\[
(e^{\mathrm{i} t_n})\simeq \operatorname{Diag}((e^{\mathrm{i} t_n})).
\]
Thus \(\mathbb{T}^{\infty}\) acts on \(\mathbb{CP}^\infty\) by Kähler morphisms via
\[
\left(e^{\mathrm{i} t_n}\right) \cdot [(z_n)] := [((e^{\mathrm{i} t_n})z_n)].
\]
The corresponding moment map is
\begin{equation}\label{eq:Tinf_momentmap_cpinf}
	\mu_{\mathbb{T}^{\infty}}: \mathbb{CP}^\infty \to (\operatorname{Lie}(\mathbb{T}^{\infty}))^*, \quad [(z_n)] \mapsto \frac{\mathrm{i}}{2} (|z_n|^2).
\end{equation}
Using the identity
\[
(|z_n|^2) = (\bar{z}_n)^T \operatorname{Diag}(1) (z_n)
\]
and noting that \(\operatorname{Diag}(1)\) defines a bounded operator, the moment map \(\mu_{\mathbb{T}^\infty}\) in \eqref{eq:Tinf_momentmap_cpinf} takes values in \((\mathrm{i} \ell^2_{\RR})^*\). By identifying \(\ell^2_{\RR}\) with its dual, we obtain the following commutative diagram:
\begin{equation}
	\begin{tikzcd}
		\mathbb{S}\left(\ell^2_{\mathbb{C}}\right) \arrow[rr, "\Psi"] \arrow[rd, "\pi"] & & \mathrm{i} \ell^2_{\mathbb{R}} \\
		& \mathbb{T}^\infty \curvearrowright \mathbb{CP}^\infty \arrow[ru, "\mu_{\mathbb{T}^\infty}"] &
	\end{tikzcd}
\end{equation}
where \(\Psi(z) := \frac{\mathrm{i}}{2} (|z_n|^2)\) and \(\pi\) denotes the Hopf fibration. Using \(\mathcal{U} \subseteq \mathbb{S}\left(\ell^2_{\mathbb{C}}\right)\) and the explicit form of the inverse of the square root map \(\varPhi\) in \Cref{t: square root map is isometry}, we see that the restriction of \(\Psi\) to \(\mathcal{U}\) is equal to \(\frac{2}{\mathrm{i}} \varPhi^{-1}\). Thus, we obtain the following.

\begin{lemma}\label{l: momentmap image}
	The restriction of \(\Psi\) to \(\mathcal{U}\) is a diffeomorphism onto \(\frac{\mathrm{i}}{2} \Delta\), and the image of the momentum map is given by
	\[
	\mu_{\mathbb{T}^\infty}(\mathbb{CP}^\infty) = \frac{\mathrm{i}}{2} \Delta.
	\]
\end{lemma}

\begin{remark}
The map \( \mu_{\mathbb{T}^{\infty}} \) can be interpreted as an \(\ell^2\)-version of the inverse of the Madelung transform in the density component, as described in \cite[Prop.~4.3]{khesin_geometry_2019}.
\end{remark}

For a choice of \( (c_n) \in \ell^2_{\RR} \), we define the Hamiltonian
\begin{equation}\label{eq defi Hamiltonian on CP}
    H_{(c_n)} : \mathbb{C}P^\infty \longrightarrow \mathbb{R}, \quad [z_n] \mapsto \left\langle (c_n), 2\bar{\i} \cdot \mu_{\mathbb{T}^\infty}([z_n]) \right\rangle_{\ell^2} = \left\langle (c_n), (|z_n|^2) \right\rangle_{\ell^2}.
\end{equation}
Since \( \Omega^{\mathrm{FS}} \) is a strong symplectic form, the Hamiltonian vector field \( X_{H_{(c_n)}} \) is uniquely determined by
\[
\Omega^{\mathrm{FS}}(X_{H_{(c_n)}}, \cdot) = \mathrm{d}H_{(c_n)}.
\]
Moreover, since \( (\mathbb{C}P^\infty, \mathcal{G}^{\mathrm{FS}}, \Omega^{\mathrm{FS}}) \) is a Kähler manifold, the gradient and Hamiltonian vector fields of \( H_{(c_n)} \) are related by multiplication with \( \mathrm{i} \), that is,
\[
X_{H_{(c_n)}}=\i \nabla H_{(c_n)}.
\]
The flow \( \varphi_{H_{(c_n)}} \) of the vector field \( X_{H_{(c_n)}} \) is called the \emph{Hamiltonian flow} of \( H_{(c_n)} \). This flow preserves the level sets
\[
\Sigma_{\kappa}:= H_{(c_n)}^{-1}(\kappa)
\]
for all energy levels \( \kappa \in \mathbb{R} \).

The following theorem shows that this system admits infinitely many Poisson-commuting first integrals and that the gradient flow interpretation from \Cref{s: 3} is compatible with the Hamiltonian picture on \(\mathbb{CP}^\infty\).

\begin{theorem}\label{t: integrability of Hamiltonian system}
The Hamiltonian system \( (\mathbb{CP}^\infty, \Omega^{\mathrm{FS}}, H_{(c_n)}) \) admits infinitely many linearly independent Poisson-commuting first integrals.

More precisely, for each \(n\in \mathbb{N}\), the Hamiltonian
\begin{equation}\label{eq: Hamiltonians on CP inf for integrability}
	H_n : \mathbb{C}P^\infty \longrightarrow \mathbb{R}, \quad [z_m] \mapsto c_n|z_n|^2
\end{equation}
is a first integral of \(H_{(c_n)}\), and for all \(k\neq n\) one has
\[
\{H_k,H_n\}_{\mathbb{C}P^\infty}=0
\qquad \text{and} \qquad
\{H_{(c_n)},H_n\}_{\mathbb{C}P^\infty}=0.
\]

Moreover, let \( (x_n)(t) \) be a gradient flow line of the restriction of the gradient flow of \( H_{(c_n)} \) on \( (\mathbb{CP}^\infty, \mathcal{G}^{\mathrm{FS}}) \) to \( \mathcal{U} \). If, in addition, \((c_n)\) is strictly monotonically decreasing, then the limit
\[
p_{\max} := \lim_{t \to \infty} \varPhi\left((x_n)(t)\right)
\]
exists and solves \eqref{eq:ell2_LPP}.
\end{theorem}

\begin{proof}
For each \( n \in \mathbb{N} \), define the sequence \( (b_m) \) by
\[
b_m:=\begin{cases}
    c_n\, &, \quad \text{for } m=n,\\
    0\, &, \quad \text{otherwise.}
\end{cases}
\]
Then the Hamiltonian \(H_n\) in \eqref{eq: Hamiltonians on CP inf for integrability} may be written as
\[
H_n([z_m]) = \sum_{m=0}^{\infty}b_m |z_m|^2 = c_n|z_n|^2.
\]
Its Hamiltonian flow \( \varphi_{H_n} \) fixes all coordinates except the \( n \)-th, where it acts by phase rotation. In particular, the Hamiltonians \(H_n\) are linearly independent and arise from the coordinate circle actions inside \(\mathbb{T}^\infty\).

We now verify the Poisson-commutation relations. It suffices to prove that the lifts \(\hat{H}_k\) and \(\hat{H}_n\) to \(\mathbb{S}(\ell^2_{\mathbb{C}})\), which naturally extend to \(\ell^2_{\mathbb{C}}\), commute and are \(\mathbb{S}^1\)-invariant.

Recall that the canonical symplectic form on \(\ell^2_{\mathbb{C}}\) is
\[
\omega_{\mathrm{can}} = \frac{\i}{2} \sum_{j=0}^{\infty} \d z_j \wedge \d \bar{z}_j,
\]
inducing the Poisson bracket
\[
\{f, g\} = 2\,\i \sum_{j=0}^{\infty} \left( \frac{\partial f}{\partial \bar{z}_j} \frac{\partial g}{\partial z_j} - \frac{\partial f}{\partial z_j} \frac{\partial g}{\partial \bar{z}_j} \right).
\]
By definition,
\[
\hat{H}_k(z) = c_k |z_k|^2 = c_k z_k \bar{z}_k,
\]
with partial derivatives
\begin{equation*}
    \frac{\partial \hat{H}_k}{\partial z_j} = c_k \bar{z}_k \delta_{jk}, \qquad
    \frac{\partial \hat{H}_k}{\partial \bar{z}_j} = c_k z_k \delta_{jk}, \quad \forall j \in \mathbb{N},
\end{equation*}
and analogously for \(\hat{H}_m\). Since \(k\neq m\), their derivatives have disjoint support, and therefore
\[
\{ \hat{H}_k, \hat{H}_m \} = 0.
\]
Using \(\mathbb{S}^1\)-invariance and the moment map description~\eqref{eq:S1_momentmap_cpinf}, it follows that
\[
\{ H_k, H_m \}_{\mathbb{C}P^{\infty}} = 0.
\]
Repeating the same argument for \(H_{(c_n)}\), we obtain
\[
\{ H_{(c_n)}, H_n \}_{\mathbb{C}P^{\infty}} = 0
\qquad \forall n\in \mathbb{N}.
\]

Finally, identify \(\mathcal{U}\) with its image in \(\mathbb{CP}^\infty\) and use \Cref{t: square root map is isometry}. Together with the identity
\[
H_{(c_n)} \circ \varPhi = F_{(c_n)} \quad \text{on } \Delta,
\]
where \( F_{(c_n)} \) is defined in \eqref{eq:def_F_cn}, the final claim follows from \Cref{prop:integral-curves}.
\end{proof}
\section{Information geometry of \(\mathrm{Dens}(M)\) on non-compact smooth manifolds}\label{S:5}

In this final section, we outline an approach to the information geometry of \(\mathrm{Dens}(M)\) on non-compact smooth manifolds \(M\), inspired by the findings in \Cref{s: 2}. The point of departure is that the square root map remains the natural object in this setting as well, and this suggests a framework that may serve as a basis for further investigation.

We begin by introducing the setting. Let \((M, g)\) be a Riemannian manifold, possibly non-compact. Denote by \(\mathrm{d}\mu\) the volume form induced by the metric \(g\), and define the space of smooth probability densities on \(M\) by
\begin{equation}\label{e def Dens}
    \mathrm{Dens}(M) := \left\{ \rho \in C^{\infty}(M) : \int_M \rho\, \mathrm{d}\mu = 1 \quad \text{and} \quad \rho > 0 \right\}.
\end{equation}
We denote the space of square-integrable functions on \(M\) with respect to \(\mathrm{d}\mu\) by \(L^2(M, \mu)\), and its inner product by \(\langle \cdot, \cdot \rangle_{L^2}\), which induces the norm \(\Vert \cdot \Vert_{L^2}\). We define
\begin{equation}\label{e: Sphere}
    \S_{L^2}^{\infty}(M):= \S^{L^2}(C^{\infty}(M,\RR)) := \left\{ f \in C^{\infty}(M,\RR) : \Vert f \Vert_{L^2} = 1 \right\}
\end{equation}
as the unit sphere of smooth functions in \(L^2(M,\mu)\).

Before proceeding, note that the two manifolds defined in \eqref{e def Dens} and \eqref{e: Sphere}, viewed as hypersurfaces of \(C^{\infty}(M)\), naturally inherit their topology from the ambient space. Furthermore, we equip \(\S_{L^2}^{\infty}(M)\) with the standard differentiable structure, i.e., the differentiable structure it inherits as a hypersurface of \(C^{\infty}(M)\). Thus its tangent space at the point \(f\) is given by
\[
T_f \S_{L^2}^{\infty}(M)= \left\{g\in C^{\infty}(M,\RR): \langle f,g \rangle_{L^2}=0\right\}. 
\]
The round metric, or equivalently the \(L^2\)-metric \(\mathcal{G}^{L^2}\), is obtained by restricting the inner product \(\langle \cdot, \cdot \rangle_{L^2}\) to the tangent space of \(\S_{L^2}^{\infty}(M)\). Thus, \((\S_{L^2}^{\infty}(M), \mathcal{G}^{L^2})\) is a weak Riemannian tame Fréchet manifold.

Next, we define the open subset of \(\S_{L^2}^{\infty}(M)\) consisting of everywhere positive functions by
\begin{equation}
    \mathbb{U}_{L^2}^{\infty}:=\{f\in \S_{L^2}^{\infty}(M): f>0\}.
\end{equation}
This allows us to recall that the square root map defines the following homeomorphism:
\begin{equation}\label{e: def square root map}
    \varPhi:  \mathrm{Dens}(M) \longrightarrow  \mathbb{U}_{L^2}^{\infty}, \quad \rho \mapsto \sqrt{\rho}.
\end{equation}
We now choose the smooth structure on \(\mathrm{Dens}(M)\) so that \eqref{e: def square root map} becomes a diffeomorphism; we call this the \emph{spherical differentiable structure} on \(\mathrm{Dens}(M)\). With this differentiable structure, the tangent space at a point \(\rho\in \mathrm{Dens}(M)\) is given by
\begin{equation}\label{e: diff strucutre on Dens}
    T_{\rho}\mathrm{Dens}(M) = \left\{ \theta \in C^{\infty}(M,\RR) : \int_M \theta\, \mathrm{d}\mu = 0 \quad \text{and} \quad \frac{\theta}{\sqrt{\rho}} \in L^2(M, \mu) \right\},
\end{equation}
which shows that this differentiable structure differs substantially from the one induced by the ambient space.  This is, in some sense, related to \cite[§2]{Friedrich1991}.

\begin{remark}
If \(M\) is compact, the condition on \(\frac{\theta}{\sqrt{\rho}}\) in \eqref{e: diff strucutre on Dens} is automatically satisfied. Therefore, in this case, the spherical differentiable structure coincides with the differentiable structure induced by the ambient space.
\end{remark}

We define the Fisher--Rao metric at the point \(\rho \in \mathrm{Dens}(M)\) in this setting by
\begin{equation}\label{e: def Fisher rao}
    \mathcal{G}^{FR}_{\rho} : T_{\rho}\mathrm{Dens}(M) \times T_{\rho}\mathrm{Dens}(M) \longrightarrow \mathbb{R}, \quad (\theta, \bar{\theta}) \mapsto \mathcal{G}^{FR}_{\rho}(\theta, \bar{\theta}) := \int_{M} \frac{\theta \cdot \bar{\theta}}{\rho}\, \mathrm{d}\mu,
\end{equation}
which is well defined with respect to the differentiable structure chosen in \eqref{e: diff strucutre on Dens}.

\begin{remark}
Analogously to \Cref{r: the right diff strcuture}, without this condition, the Fisher--Rao metric \eqref{e: def Fisher rao} is not well defined. Even under decay assumptions analogous to those in \cite[§3.5]{MichorMumfordZoo}, one can construct rapidly decaying functions in \(\mathrm{Dens}(\RR)\) such that the Fisher--Rao metric fails to be finite on some tangent vectors.
\end{remark}

Next, we show that in this framework the square root map is again an isometry.

\begin{theorem}\label{thM:square_roor_isom_for_non_closed_mnfds}
If \(\mathrm{Dens}(M)\) is equipped with the spherical differentiable structure given in \eqref{e: diff strucutre on Dens}, then the square root map
\[
    \varPhi:  \left(\mathrm{Dens}(M), \mathcal{G}^{FR}\right) \longrightarrow \left(\mathbb{U}_{L^2}^{\infty}, \mathcal{G}^{L^2}\right), \quad \rho \mapsto \sqrt{\rho}
\]
is an isometry.
\end{theorem}

\begin{remark}
\Cref{thM:square_roor_isom_for_non_closed_mnfds} extends the result~\cite[Thm.~3.1]{khesin_GAFA} from compact to non-compact manifolds.
\end{remark}

\begin{proof}
The proof follows the lines of \cite[Thm.~3.1]{khesin_GAFA}, using the spherical differentiable structure introduced above and the fact that the square root map identifies the Fisher--Rao metric with the \(L^2\)-metric on \(\mathbb{U}_{L^2}^{\infty}\).
\end{proof}

The perspective developed here suggests that the square root map may provide a useful framework for studying information geometry on non-compact domains more systematically. In particular, it is plausible that constructions which rely heavily on the fact that the square root map is an isometry can be extended from the \(\ell^2\)-setting to \(\mathrm{Dens}(M)\) on non-compact manifolds. It would therefore be interesting to investigate whether results such as \Cref{prop:integral-curves}, \Cref{prop:e_geodesics_equation}, and \Cref{t: integrability of Hamiltonian system} admit analogues in this broader setting.\\

\noindent
\textbf{Acknowledgements.:}
The author thanks P. Albers, M. Bleher, J. Cassel, Y. Elshiaty, and F. Schlindwein for valuable discussions.\\
The author acknowledges funding by the Deutsche Forschungsgemeinschaft (DFG, German Research Foundation) – 281869850 (RTG 2229), 390900948 (EXC-2181/1) and 281071066 (TRR 191).\\
The author would like to acknowledge the excellent working
conditions and interactions at Erwin Schrödinger International
Institute for Mathematics and Physics, Vienna, during the thematic
programme \emph{``Infinite-dimensional Geometry: Theory and Applications"}
where part of this work was completed.\\
The author gratefully acknowledges the anonymous referees for their insightful comments and suggestions, which have improved the quality of the paper.
\\

\noindent\textbf{Data availability} In this research no data is processed.\\
\noindent\\ 
\textbf{Declarations}\\\\
\textbf{Conflict of interest:} The author states that there is no conflict of interest.
%\appendix 
%\section{Proof of \Cref{prop:integral-curves}}\label{Appendix proof Prop}

%\section{Proof of \Cref{t: integrability of Hamiltonian system}}\label{Appendix proof of integrability of Hamiltonian system}

\bibliographystyle{abbrv}
	\bibliography{literature}
\end{document}